\documentclass{amsart}

\usepackage{macros}
\standardmargins\standardrenews\standardmakes\standardlabeling

\PassOptionsToPackage{hyphens}{url}
\RequirePackage[breaklinks,hidelinks]{hyperref}
\RequirePackage{color}
\RequirePackage[usenames,dvipsnames]{xcolor}
\RequirePackage{pgf, tikz}
\usetikzlibrary{arrows, patterns}

\allowcomments{\comryan}{RB}{Ryan}{blue}
\newcommand{\rad}{\,\text{rad}}
\newcommand{\thi}{\text{thi}}
\newcommand{\ar}{r}
\renewcommand{\hyp}{\LL}
\newcommand{\arith}{L_{\text{AP}}}
\newcommand{\sep}{E}
\newcommand{\turn}{m}
\renewcommand{\index}{n}

\draftfalse

\begin{document}
\title[Quantitative results using variants of Schmidt's game]{Quantitative results using variants of Schmidt's game: Dimension bounds, arithmetic progressions, and more}

\authorryan\authorlior\authordavid

\begin{Abstract}
Schmidt's game is generally used to deduce qualitative information about the Hausdorff dimensions of fractal sets and their intersections. However, one can also ask about quantitative versions of the properties of winning sets. In this paper we show that such quantitative information has applications to various questions including:
\begin{itemize}
\item What is the maximal length of an arithmetic progression on the ``middle $\epsilon$'' Cantor set?
\item What is the smallest $n$ such that there is some element of the ternary Cantor set whose continued fraction partial quotients are all $\leq n$?
\item What is the Hausdorff dimension of the set of $\epsilon$-badly approximable numbers on the Cantor set?
\end{itemize}
We show that a variant of Schmidt's game known as the \emph{potential game} is capable of providing better bounds on the answers to these questions than the classical Schmidt's game. We also use the potential game to provide a new proof of an important lemma in the classical proof of the existence of Hall's Ray.
\end{Abstract}

\maketitle

\tableofcontents

\section{Introduction}

As motivation, we begin by considering three questions which initially appear to be unrelated, but whose answers turn out to have a deep connection.

\begin{question}
\label{q1}
For each $0 < \epsilon < 1$, let $M_\epsilon$ be the \emph{middle-$\epsilon$ Cantor set} obtained by starting with the interval $[0,1]$ and repeatedly deleting from each interval appearing in the construction the middle open interval of relative length $\epsilon$. As $\epsilon \to 0$, the sets $M_\epsilon$ are getting ``larger'' in the sense that their Hausdorff dimensions tend to 1. Do they also get ``larger'' in the sense of containing longer and longer arithmetic progressions as $\epsilon\to 0$? How does the length of the longest arithmetic progression in $M_\epsilon$ behave as $\epsilon \to 0$?
\end{question}

\begin{question}
\label{q3}
What is the Hausdorff dimension of the set of $\epsilon$-badly approximable vectors
\[
\BA_d(\epsilon) \df \{\xx\in\R^d : \forall \pp/q \in \Q^d \;\; |\xx-\pp/q| > \epsilon q^{-\frac{d+1}{d}}\}?
\]
Here $|\cdot|$ denotes a fixed norm on $\R^d$.
\end{question}

\begin{question}
\label{q2}
For each $n\in\N$, let $F_n$ denote the set of irrational numbers in $(0,1)$ whose continued fraction partial quotients are all $\leq n$.\Footnote{The \emph{continued fraction expansion} of an irrational number is the unique expression
\[
a_0 + \cfrac1{a_1+\cfrac1{a_2+\ddots}}
\]
with $a_0\in\Z$, $a_1,a_2,\ldots\in\N$, whose value is equal to that number. The numbers $a_1,a_2,\ldots$ are called the \emph{partial quotients}.} The union of $F_n$ over all $n$ is the set of badly approximable numbers in $(0,1)$, i.e. $(0,1)\cap \bigcup_{\epsilon>0} \BA_1(\epsilon)$, which is known to have full dimension in the ternary Cantor set $C = M_{1/3}$, so in particular we have $F_n \cap C \neq \emptyset$ for all sufficiently large $n$. What is the smallest $n$ for which $F_n\cap C \neq \emptyset$?
\end{question}

What these questions have in common is that they can all be (partially) answered using Schmidt's game, a technique for proving lower bounds on the Hausdorff dimensions of certain sets known as ``winning sets'', as well as on the dimensions of their intersections with other winning sets and with various nice fractals. In particular, the class of winning sets (see e.g. \cite{Schmidt1} for the definition) has the following properties:
\begin{itemize}
\item[(a)] The class of winning sets is invariant under bi-Lipschitz maps and in particular under translations.
\item[(b)] The intersection of finitely many winning sets is winning.\Footnote{If $\alpha > 0$ is fixed, then the intersection of countably many $\alpha$-winning sets is $\alpha$-winning (see \cite{Schmidt1} for the definition of $\alpha$-winning). But if $\alpha$ is not fixed, then it may only be possible to intersect finitely many winning sets.}
\item[(c)] Winning sets in $\R^d$ have full Hausdorff dimension and in particular are nonempty.
\item[(d)] The set of badly approximable numbers is winning, both in $\R$ and on the Cantor set.
\end{itemize}
These properties already hint at why Schmidt's game might be relevant to Questions \ref{q1}--\ref{q2}. Namely, properties (a), (b), and (c) imply that any winning subset of $\R$ contains arbitrarily long arithmetic progressions (since if $S$ is winning, then for any $t,k$ the set $\bigcap_{i=0}^{k-1} (S-it)$ is winning and therefore nonempty), and properties (c) and (d) imply that the set of badly approximable numbers has full Hausdorff dimension both in $\R$ and on the Cantor set. The middle-$\epsilon$ Cantor set $M_\epsilon$ is not winning, but by showing that it is ``approximately winning'' in some quantitative sense, we will end up getting a lower bound on the maximal length of arithmetic progressions it contains, thus addressing Question \ref{q1}. Similarly, the set $\BA_d(\epsilon)$ of $\epsilon$-badly approximable points in $\R^d$ is not winning, but since the union $\bigcup_{\epsilon > 0} \BA_d(\epsilon)$ is the set of all badly approximable points, it is winning and thus the dimension of $\BA_d(\epsilon)$ tends to $d$ as $\epsilon$ tends to $0$. Again, showing that $\BA_d(\epsilon)$ is ``approximately winning'' will yield a lower bound on its Hausdorff dimension, thus addressing Question \ref{q3}. Using a known relationship between $\BA_1(\epsilon)$ and $F_n$, this yields lower bounds on the Hausdorff dimensions of both $F_n$ and $F_n\cap C$. For $n$ for which the second of these lower bounds is positive, we have $F_n\cap C \neq \emptyset$, which addresses Question \ref{q2}.

To make the above paragraph rigorous, we will need to have a clear notion of what it means for a set to be ``approximately winning'' in a quantitative sense. One idea is to use Schmidt's original definition of ``$(\alpha,\beta)$-winning'' sets (see \cite{Schmidt1}) as a quantitative approximation of winning sets. This is particularly natural because the notion of being $(\alpha,\beta)$-winning is the basis of Schmidt's definition of the class of winning sets. However, it turns out that the class of $(\alpha,\beta)$-winning sets is not very nice from a quantitative point of view (see Remark \ref{remarksomuchworse}). Thus, we will instead consider two variants of Schmidt's game, the \emph{absolute} game introduced by McMullen \cite{McMullen_absolute_winning} and the \emph{potential} game introduced in \cite[Appendix C]{FSU4}. The natural notions of ``approximately winning'' for these games turn out to be more suited to proving quantitative results.

\section{Main results}

Before listing our main results, we will state what we expect the answers to Questions \ref{q1}--\ref{q2} to be based on the following heuristic: if $S_1,S_2$ are two fractal subsets of $\R^d$, then if $S_1$ and $S_2$ are ``independent'' we expect that
\begin{equation}
\label{independent1}
\HD(S_1\cap S_2) = \max(0,\HD(S_1) + \HD(S_2) - d),
\end{equation}
or equivalently
\begin{equation}
\label{independent2}
\codim_H(S_1\cap S_2) = \min(d,\codim_H(S_1) + \codim_H(S_2)),
\end{equation}
where $\codim_H(S) = d - \HD(S)$. This is roughly because if we divide $[0,1]^d$ into $N^d$ blocks of size $1/N$, then we should expect that $S_1$ will intersect $N^{\HD(S_1)}$ of these blocks and $S_2$ will intersect $N^{\HD(S_2)}$ of them, so if $S_1$ and $S_2$ are independent then we should expect
\[
\dfrac{N^{\HD(S_1)}N^{\HD(S_2)}}{N^d} = N^{\HD(S_1) + \HD(S_2) - d}
\]
blocks to be intersected by both $S_1$ and $S_2$. We can expect that most such blocks will also intersect $S_1\cap S_2$. Of course, if the exponent is negative then we should expect that $S_1\cap S_1 = \emptyset$ and in particular $\HD(S_1\cap S_2) = 0$.

To use this heuristic to estimate the maximal length of an arithmetic progression on $M_\epsilon$, note that if $\{a,a+t,\ldots,a+(k-1)t\}$ is such an arithmetic progression, then we have
\[
\bigcap_{i = 0}^{k-1} (M_\epsilon - it) \neq \emptyset.
\]
If the sets $M_\epsilon,M_\epsilon-t,\ldots,M_\epsilon-(k-1)t$ are independent, then we expect the Hausdorff dimension of their intersection to be
\[
\max(0,1-k\codim_H(M_\epsilon)),
\]
which is positive if and only if $k < 1/\codim_H(M_\epsilon)$. Since $\codim_H(M_\epsilon) \sim \epsilon$, this means that we expect the maximal length of an arithmetic progression on $M_\epsilon$ to be approximately $1/\epsilon$.\Footnote{There is an additional degree of freedom with respect to $t$ that this heuristic argument does not take into account, but its contribution to the expected maximal length of an arithmetic progression is not very significant.} Similarly, since $\codim_H(F_n) \sim 1/n$, we expect the maximal length of an arithmetic progression on $F_n$ to be approximately $n$. We are able to prove the following bounds rigorously:

\begin{theorem}
\label{theoremarithbounds}
Let $\arith(S)$ denote the maximal length of an arithmetic progression in the set $S$. For all $\epsilon > 0$ sufficiently small and $n\in\N$ sufficiently large, we have
\begin{align}
\label{arithbounds1}
\frac{1/\epsilon}{\log(1/\epsilon)} \lesssim \arith(M_\epsilon) &\leq 1/\epsilon+1\\ \label{arithbounds2}
\frac{n}{\log(n)} \lesssim \arith(F_n) &\lesssim n^2.
\end{align}
\end{theorem}
Here and hereafter, $A \lesssim B$ means that there exists a constant $K$ (called the \emph{implied constant}) such that $A \leq KB$, and $A \asymp B$ means $A \lesssim B \lesssim A$.

\begin{remark*}
When $1/\epsilon$ is an integer, the lower bound of \eqref{arithbounds1} was first proven by Jon Chaika, see \cite{Chaika}.
\end{remark*}

Theorem \ref{theoremarithbounds} does not give any information about the implied constants of \eqref{arithbounds1} and \eqref{arithbounds2}, so for example it cannot tell us how small $\epsilon$ has to be before we can be sure that $M_\epsilon$ contains an arithmetic progression of length 3 (i.e. a nontrivial arithmetic progression). However, using similar techniques we can show:

\begin{theorem}
\label{theoremM49}
For all $0 < \epsilon \leq 1/49$, we have $\arith(M_\epsilon) \geq 3$, i.e. $M_\epsilon$ contains an arithmetic progression of length $3$. Also, $F_{49}$ contains an arithmetic progression of length $3$ (and thus so does $F_n$ for all $n\geq 49$).
\end{theorem}
\begin{remark*}
It was pointed out to us by Pablo Shmerkin that one can get a better result using Newhouse's gap lemma \cite[p.107]{Newhouse2}, namely that $\arith(M_\epsilon) \geq 4$ for all $0 < \epsilon \leq 1/3$.\Footnote{Namely, Newhouse's gap lemma implies that there exists $t\in (M_\epsilon-\tfrac12)\cap \tfrac13(M_\epsilon-\tfrac12)$, and then $\big\{\tfrac12-3t,\tfrac12-t,\tfrac12+t,\tfrac12+3t\big\}$ is an arithmetic progression in $M_\epsilon$ of length $4$.} Note that for all $\epsilon > 1/3$, we have $\arith(M_\epsilon) = 2$ (the proof is similar to the proof of the upper bound of \eqref{arithbounds1}).
\end{remark*}

Question \ref{q2} can also be addressed via the independence assumption \eqref{independent1}. Namely, we have $\HD(F_2) \sim 0.531$ \cite[Theorem 10]{Good} and $\HD(C) = \frac{\log(2)}{\log(3)} \sim 0.631$, so we should expect
\[
\HD(F_2\cap C) \sim 0.531+0.631-1=0.162 > 0,
\]
and in particular $F_2\cap C \neq \emptyset$. This guess appears to be confirmed by computer estimates, which give $\HD(F_2\cap C) \sim 0.14$.\Footnote{We estimated the dimension of $F_2\cap C$ by searching for a disjoint cover of $F_2$ by intervals of the form $I_\omega = [[0;\omega,1],[0;\omega,3]]$ or $I_\omega = [[0;\omega,3],[0;\omega,1]]$ (see \eqref{cfracdef} for the notation), where $\omega$ is a finite word in the alphabet $\{1,2\}$, such that either
\begin{itemize}
\item[(A)] $I_\omega\cap C \neq \emptyset$ and $|I_\omega| < \epsilon \leq |I_{\omega'}|$, where $\omega'$ is the word resulting from deleting the last letter of $\omega$; or
\item[(B)] $I_\omega\cap C = \emptyset$.
\end{itemize}
Here $\epsilon > 0$ is a free parameter determining the accuracy of the computation. We then used the heuristic estimate
\[
\HD(F_2\cap C) \sim \frac{\log(N_\epsilon)}{-\log(\epsilon)},
\]
where $N_\epsilon$ is the calculated number of intervals of type (A). The right-hand side varies with respect to $\epsilon$ but remains within the range $[0.13,0.15]$ for $\epsilon\in [10^{-18},10^{-8}]$.} It appears quite ambitious to prove such a statement, but we can prove the following weaker one:

\begin{theorem}
\label{theoremCF19}
For each $n\in\N$, let $F_n$ denote the set of irrational numbers in $(0,1)$ whose continued fraction partial quotients are all $\leq n$, and let $C$ denote the ternary Cantor set. Then $F_{19}\cap C \neq \emptyset$.
\end{theorem}

Yann Bugeaud pointed out to us that his Folding Lemma \cite[Theorem D.3]{Bugeaud4} can be used to prove the stronger result that $F_9\cap C \neq \emptyset$.\Footnote{The proof is as follows. For the notation see \eqref{cfracdef}. Call a rational $p/q$ \emph{good} if:
\begin{itemize}
\item $q$ is a power of $3$, and
\item $p/q = [0;1,1,a_3...,a_h]$ with $h \geq 4$, $a_h \geq 2$, $h$ odd,  and $a_i \leq 3$ for all $i = 3,\ldots,h$.
\end{itemize}
By direct calculation, the rational $17/27 = [0;1,1,1,2,3]$ is good. Moreover, by the Folding Lemma \cite[Theorem D.3]{Bugeaud4}, if $p/q$ is good then so is $f(p/q) := p/q - 1/3q^2$. Thus $f^n(17/27)$ is good for all $n$ and thus $x := \lim_{n\to\infty} f^n(17/27)$ is in $F_3 \subset \BA(\frac15)$ (cf. \cite[Theorems 6 and 9]{Khinchin_book} for the subset relation). Let $y=2-2x = \sum_{k\geq 2} 2/3^{2^k-1} \in C$. Since $x\in \BA_1(\frac15)$, we have $y \in \BA_1(\frac1{10}) \subset F_9$. So $y\in F_9\cap C$.}

Question \ref{q3} is different from our other two questions in that a fairly precise answer is already known: we have
\begin{equation}
\label{kd}
\codim_H(\BA_d(\epsilon)) \sim k_d \epsilon^d,
\end{equation}
where $k_d$ is an explicit constant of proportionality and $A\sim B$ means that $A/B\to 1$ as $\epsilon \to 0$ \cite{Simmons5}\Footnote{Note that $\kappa$ in the notation of \cite{Simmons5}, and $c$ in the notation of \cite{BroderickKleinbock}, are both equal to $\epsilon^d$ in our notation (and $c^n$ in the notation of \cite{Weil1}).} (see also \cite{Kurzweil,Hensley, BroderickKleinbock, Weil1}). However, from a historical perspective the first proof that $\HD(\BA_d(\epsilon)) \to d$ as $\epsilon \to 0$ is Schmidt's proof using his eponymous game \cite{Schmidt2}, so it is interesting to ask what the best bound is that can be proven using Schmidt's game or its variants. Kleinbock and the first-named author used a variant of Schmidt's game (specifically the hyperplane absolute game) to prove that
\[
\codim_H(\BA_d(\epsilon)) \lesssim \frac{\epsilon^{1/2}}{\log(1/\epsilon)},
\]
see \cite[Theorem 1.2]{BroderickKleinbock}. We are able to improve their result by proving the following using the hyperplane potential game instead of the hyperplane absolute game:

\begin{theorem}
\label{theoremBAdcodim}
For all $\epsilon > 0$ we have
\[
\codim_H(\BA_d(\epsilon)) \lesssim \epsilon.
\]
\end{theorem}
Note that this upper bound is the correct order of magnitude when $d = 1$ but not for larger $d$.

We can also ask about the intersection of $\BA_d(\epsilon)$ with a fractal set. Recall that a compact set $J \subset \R^d$ is called \emph{Ahlfors regular of dimension $\delta$} if there exists a measure $\mu$ with topological support equal to $J$ such that for all $\xx\in J$ and $0 < \rho \leq 1$, we have
\[
C^{-1} \rho^\delta \leq \mu(B(\xx,\rho)) \leq C\rho^\delta
\] 
where $C$ is an absolute constant. It was proven in \cite{Fishman} that if $J \subset \R$ is any Ahlfors regular set, then the union $\BA_1 \df \bigcup_{\epsilon > 0} \BA_1(\epsilon)$ has full dimension in $J$. We can prove the following quantitative version of this result:

\begin{theorem}
\label{theoremKepsilondelta}
Let $J \subset \R$ be an Ahlfors regular set of dimension $\delta > 0$. Then for all $\epsilon > 0$, we have
\begin{align*}
\HD(M_\epsilon\cap J) &\geq \delta - K \epsilon^\delta\\
\HD(\BA_1(\epsilon)\cap J) &\geq \delta - K \epsilon^\delta
\end{align*}
where $K$ is a constant depending on $J$.
\end{theorem}

Note that if $\BA_1(\epsilon)$ and $J$ are independent in the sense of \eqref{independent1}, then we have
\[
\HD(\BA_1(\epsilon)\cap J) = \delta - \codim_H(\BA_1(\epsilon)) = \delta - k_1 \epsilon + o(\epsilon),
\]
where $k_1$ is as in \eqref{kd}. This means that Theorem \ref{theoremKepsilondelta} can be considered close to optimal when $\delta = 1$ but not for smaller $\delta$.

We also prove a higher-dimensional analogue of Theorem \ref{theoremKepsilondelta}, where the set $J$ needs to satisfy an additional condition known as \emph{absolute decay}; see Sections \ref{sectionpotentialHD}-\ref{sectionpotentialapplications} for details.

Quantitative versions of variants of Schmidt's game also have applications to the proof of the existence of Hall's Ray. We recall that Hall's Ray is a set of the form $\CO{t_0}\infty$ such that for all $t\in \CO{t_0}\infty$, there exists $x\in\R$ such that
\[
\limsup_{q\to\infty} \frac{1}{q^2 \|qx\|} = t,
\]
where $\|\cdot\|$ denotes distance to the nearest integer. The proof of the existence of Hall's Ray proceeds via a lemma stating that $F_4 + F_4 = [\sqrt 2 - 1,4\sqrt 2 - 4]$, see e.g. \cite[Chapter 4]{CusickFlahive}. The proof of this lemma is ordinarily via a ``hands-on'' argument, but we are able to prove the following weaker version of the lemma in a more conceptual way using the absolute game:

\begin{theorem}
\label{theoremF49}
$F_{49}+F_{49} \supset [1/6,11/6]$.
\end{theorem}

This weaker version is still sufficient to prove the existence of Hall's Ray, although it would yield a worse bound for $t_0$ than a proof using the original lemma.\\

\textbf{Acknowledgements.} We thank Vitaly Bergelson for suggesting the Question \ref{q1} to us, and Yann Bugeaud for suggesting Question \ref{q2}, as well as for pointing out the relevance of his Folding Lemma to this question. We thank Jon Chaika for helpful discussions including providing us with an early draft of \cite{Chaika}. We thank Pablo Shmerkin for his comments regarding the Newhouse gap lemma. The second-named author was supported in part by the Simons Foundation grant \#245708. The third-named author was supported by the EPSRC Programme Grant EP/J018260/1. The authors thank an anonymous referee for helpful comments.

\section{The absolute game and its applications}
\label{sectionabsolute}

We will prove most of our results using two different variations of Schmidt's game: the absolute game and the potential game. In this section we define the absolute game and use it to prove Theorems \ref{theoremM49}, \ref{theoremCF19}, and \ref{theoremF49}. The version of the absolute game that we give below is slightly different from the standard definition as found in e.g. \cite{McMullen_absolute_winning}. Specifically, the following things are different:
\begin{itemize}
\item We introduce a parameter $\rho$ limiting Bob's initial move by preventing him from playing too small a ball at first.
\item We use two separate parameters $\alpha,\beta$ to limit Alice and Bob's moves during the game, rather than a single parameter $\beta$ as in the classical definition of the absolute game.
\item We introduce a parameter $k$ allowing Alice to delete a fixed finite number of balls rather than a single ball.
\item We use the convention that if a player has no legal moves, he loses. Note that this convention means that the implication ``winning implies nonempty'' will only be true for certain sets of parameters.
\end{itemize}
The first three changes are for the purpose of recording more precise quantitative information about winning sets, while the last change allows for more elegant general statements (cf. Remark \ref{remarkjustifyconvention}), as well as making the absolute game more similar to the potential game that we define in Section \ref{sectionpotential}.

\begin{definition}
\label{Absolute}
Let $X$ be a complete metric space. Given $\alpha,\beta,\rho > 0$ and $k\in\N$, Alice and Bob play the \emph{$(\alpha,\beta,\rho,k)$-absolute game} as follows:
\begin{itemize}
\item The turn order is alternating, with Bob playing first. Thus, Alice's $\turn$th turn occurs after Bob's $\turn$th turn and before Bob's $(\turn+1)$st turn. It is thought of as a response to Bob's $\turn$th move.
\item On the $\turn$th turn, Bob plays a closed ball $B_\turn = B(x_\turn,\rho_\turn)$, and Alice responds by choosing at most $k$ different closed balls $A_\turn^{(i)}$, each of radius $\leq\alpha\rho_\turn$. She is thought of as ``deleting'' these balls.
\item On the first (0th) turn, Bob's ball $B_0 = B(x_0,\rho_0)$ is required to satisfy
\begin{equation}
\label{bobrulesabs1}
\rho_0 \geq \rho,
\end{equation}
while on subsequent turns his ball $B_{\turn + 1} = B(x_{\turn+1},\rho_{\turn+1})$ is required to satisfy
\begin{equation}
\label{bobrulesabs}
\rho_{\turn + 1}\geq \beta \rho_\turn \text{ and } B_{\turn + 1} \subset B_\turn \butnot \bigcup_i A_\turn^{(i)}.
\end{equation}
If Bob cannot choose a ball consistent with these rules, then he loses automatically.
\end{itemize}
After infinitely many turns have passed, Bob has chosen an infinite descending sequence of balls
\[
B_0 \supset B_1 \supset \cdots
\]
If the radii of these balls do not tend to zero, then Alice is said to win by default. Otherwise, the balls intersect at a unique point $x_\infty\in X$, which is called the \emph{outcome} of the game.

Now fix a set $S\subset X$, called the \emph{target set}. Suppose that Alice has a strategy guaranteeing that if she does not win by default, then $x_\infty \in S$. Then the set $S$  is called \emph{$(\alpha,\beta,\rho,k)$-absolute winning}. A set is called \emph{$(\alpha,\beta,\rho)$-absolute winning} if it is $(\alpha,\beta,\rho,1)$-absolute winning. If a set is $(\alpha,\beta,\rho)$-absolute winning for all $\alpha,\beta,\rho > 0$, then it is called \emph{absolute winning}.
\end{definition}

Note that the last part of this definition, defining the term ``absolute winning'' with no parameters, defines a concept identical to the standard concept of absolute winning, even though the intermediate concepts are different.

We warn the reader that despite the terminology, an $(\alpha,\beta,\rho,k)$-absolute winning set is not necessarily absolute winning. To prevent confusion, we will sometimes call an absolute winning set \emph{strongly} absolute winning, and a set which is $(\alpha,\beta,\rho,k)$-absolute winning for a suitable choice of $\alpha,\beta,\rho,k$ \emph{weakly} or \emph{approximately} absolute winning.

Finally, note that if $\alpha \geq 1$, then Alice can win in one turn by simply choosing $A_m^{(1)} = B_m$, regardless of the target set.

\subsection{Properties of the absolute game}
The most basic properties of the absolute game are the finite intersection property, monotonicity with respect to the parameters, and invariance under similarities.

\begin{proposition}[Finite intersection property]
\label{propositionintersection1}
Let $J$ be a finite index set, and for each $j\in J$ let $S_j$ be $(\alpha,\beta,\rho,k_j)$-absolute winning. Let $k = \sum_{j\in J} k_j$. Then $S = \bigcap_{j\in J} S_j$ is $(\alpha,\beta,\rho,k)$-absolute winning.
\end{proposition}
\begin{proof}
Alice can play all of the strategies corresponding to the sets $S_j$ simultaneously by deleting all of the balls she is supposed to delete in each of the individual games. The formula $k = \sum_{j\in J} k_j$ guarantees that this strategy will be legal to play.
\end{proof}

\begin{proposition}[Monotonicity]
\label{propositionmonotonicity1}
If $S$ is $(\alpha,\beta,\rho,k)$-absolute winning and $\alpha\leq \w\alpha$, $\beta \leq \w\beta$, $\rho \leq \w\rho$, and $k\leq \w k$, then $S$ is also $(\w\alpha,\w\beta,\w\rho,\w k)$-absolute winning.
\end{proposition}
\begin{proof}
Switching from the $(\alpha,\beta,\rho,k)$-game to the $(\w\alpha,\w\beta,\w\rho,\w k)$-game increases Alice's set of legal moves while decreasing Bob's, making it easier for Alice to win.
\end{proof}

\begin{proposition}[Invariance under similarities]
\label{propositioninvariance1}
Let $f:X\to Y$ be a bijection that satisfies
\[
\dist(f(x),f(y)) = \lambda \dist(x,y) \all x,y\in X.
\]
Then a set $S \subset X$ is $(\alpha,\beta,\rho,k)$-absolute winning if and only if the set $f(S) \subset Y$ is $(\alpha,\beta,\lambda\rho,k)$-absolute winning.
\end{proposition}
\begin{proof}
The $(\alpha,\beta,\lambda\rho,k)$-game on $Y$ with target set $f(S)$ is a disguised version of the $(\alpha,\beta,\rho,k)$-game on $X$ with target set $S$, with the moves $B_\turn,A_\turn$ in the latter game corresponding to the moves $f(B_\turn),f(A_\turn)$ in the former game.
\end{proof}

\begin{remark}
\label{remarkjustifyconvention}
In the proof of Proposition \ref{propositionintersection1} we crucially used the fact that Bob is considered to lose if he cannot play. This is because although the formula $k = \sum_{j\in J} k_j$ proves that the strategy described will be legal for Alice to play, it does not show that Bob necessarily has any legal responses to it. A similar comment applies to the proof of Proposition \ref{propositionmonotonicity1}. The question of what circumstances guarantee that Bob has legal responses will be dealt with in Lemma \ref{lemmanonempty} below.
\end{remark}

\begin{remark}
\label{remarksomuchworse}
The proof of Proposition \ref{propositionintersection1}, which proceeds by combining different strategies used on the same turn of two different games, is very different from the proof of the countable intersection property of the classical Schmidt's game (see \cite[Theorem 2]{Schmidt1}), which proceeds by splicing different turns of different games together to get a sequence of turns in a new game with different parameters. This difference is in fact the key advantage of the absolute game over the classical Schmidt's game for quantitative purposes. By modifying the argument given in \cite[Theorem 2]{Schmidt1} one can use the classical Schmidt's game to prove results such as $\arith(M_\epsilon) \gtrsim \log(1/\epsilon)$ and $\arith(F_n) \gtrsim \log(n)$, but these bounds are so much worse than the ones appearing in Theorem \ref{theoremarithbounds} that we do not include their proof.
\end{remark}

A key fact about (strongly) absolute winning sets is that in a sufficiently nice (e.g. Ahlfors regular) space, they have full Hausdorff dimension and in particular are nonempty. (This follows from the corresponding property for classical winning sets, see e.g. \cite[Theorem 3.1]{Fishman}.) The following lemma is a quantitative version of this property:

\begin{lemma}
\label{lemmanonempty}
Let $S \subset X$ be $(\alpha,\beta,\rho,k)$-absolute winning, and suppose that every ball $B\subset X$ contains $N > k$ disjoint subballs of radius $\beta\rad(B)$ separated by distances of at least $2\alpha\rad(B)$. Let $B_0$ be a ball of radius $\geq \rho$. Then $S\cap B_0 \neq \emptyset$, and
\begin{equation}
\label{HDbound}
\HD(S\cap B_0) \geq \frac{\log(N-k)}{-\log(\beta)}\cdot
\end{equation}
\end{lemma}

For our purposes, the important part of this lemma is the assertion that $S\cap B_0 \neq \emptyset$, but we include the bound \eqref{HDbound} because it can be proven easily. Note that the formula $S\cap B_0 \neq \emptyset$ follows from \eqref{HDbound} when $N > k+1$ but not when $N=k+1$.
\begin{proof}
For each ball $B\subset X$, let $f_1(B),\ldots,f_N(B)$ denote the disjoint subballs of radius $\beta\rad(B)$ guaranteed by the assumption. We will consider strategies for Bob that begin by playing $B_0$ and continue playing using the functions $f_1,\ldots,f_N$; that is, on the turn after playing a ball $B$ Bob will play one of the balls $f_1(B),\ldots,f_N(B)$. Some of these strategies are ruled out by the rules of the absolute game, but the separation hypothesis guarantees that at most $k$ of them are ruled out, and thus at least $N-k$ are left. In particular, since $N > k$ Bob always has at least one legal play, so it is possible for him to play the entire game legally starting with the move $B_0$. The outcome of the corresponding game is a member of $S\cap B_0$, and in particular this set is nonempty.

To demonstrate \eqref{HDbound}, we observe that since Bob actually had at least $N-k$ legal moves at each stage, the resulting Cantor set consisting of all possible outcomes of games where Bob plays according to the strategies described above is produced by a branching construction in which each ball of radius $r$ has at least $N-k$ children of radius $\beta r$. It is well-known (see e.g. \cite{Beardon4}) that the Hausdorff dimension of a Cantor set constructed in this way is at least $\frac{\log(N-k)}{-\log(\beta)}$. This completes the proof.
\end{proof}

\begin{corollary}
\label{corollarynonempty}
Let $S \subset \R$ be $(\alpha,\beta,\rho,k)$-absolute winning, and suppose that $k\alpha + (k+1)\beta \leq 1$. Let $I_0$ be an interval of length $\geq 2\rho$. Then $S\cap I_0 \neq \emptyset$.
\end{corollary}
\begin{proof}
Every interval $I$ in $\R$ can be subdivided into $(k+1)$ equally spaced intervals of length $\beta|I|$, such that the left endpoint of the leftmost interval is equal to the left endpoint of $I$, and the right endpoint of the rightmost interval is equal to the right endpoint of $I$. There are $k$ gaps between these intervals, so the common gap size $\epsilon$ satisfies $k\epsilon + (k+1)\beta|I| = |I|$. By assumption $k\alpha + (k+1)\beta \leq 1$, so we have $\alpha|I| \leq \epsilon$, i.e. the gap size is at least $\alpha|I| = 2\alpha\rad(I)$. Thus the hypotheses of Lemma \ref{lemmanonempty} are satisfied.
\end{proof}

\subsection{Applications}
We begin by showing that the sets $M_\epsilon$ and $\BA_1(\epsilon)$ are weakly absolute winning.

\begin{lemma}
\label{lemmaMepsilon}
For all $0 < \epsilon < 1$ and $0 < \beta < 1$, the set $M_\epsilon \cup (-\infty, 0) \cup (1, \infty)$ is $(\alpha,\beta,\rho)$-absolute winning where
\[
(\alpha,\beta,\rho) = \left(\tfrac{2\epsilon}{1-\epsilon} \beta^{-1},\beta,\tfrac{1-\epsilon}{2}\beta/2\right).
\]
\end{lemma}
\begin{proof}
To describe Alice's strategy, let $B$ be a move for Bob and we will describe Alice's response. Let $\lambda = \frac{1-\epsilon}{2}$, and let $n\geq 0$ be the largest integer such that $\lambda^{n+1} \geq |B|$, if such an integer exists. At the $(n+1)$st stage of the construction of $M_\epsilon$, all remaining intervals are of length $\lambda^{n+1}$, which means that the distances between the removed intervals of the first $n$ stages are all at least $\lambda^{n+1}$. So $B$ intersects at most one of these intervals, and Alice's strategy is to remove this interval if it is legal to do so, and otherwise to not delete anything. If the integer $n$ does not exist, then Alice does not delete anything.

To show that this strategy is winning, we must show that if Alice did not win by default, then the outcome of the game $x_\infty$ is in $M_\epsilon \cup (-\infty, 0) \cup (1, \infty)$. By contradiction, suppose that $x_\infty$ exists but is not in $M_\epsilon \cup (-\infty, 0) \cup (1, \infty)$. Then $x_\infty$ is in some interval $I$ that was removed during the construction of $M_\epsilon$. We will show that Alice deleted $I$ at some stage of the game, which contradicts the fact that $x_\infty$ was obtained by a sequence of legal plays for Bob.

Indeed, let $n\geq 0$ be the stage of the construction of $M_\epsilon$ at which $I$ was removed, so that $|I| = \lambda^n \epsilon$, and let $\turn\geq 0$ be the smallest integer such that $\lambda^{n+1} \geq |B_\turn|$. If $\turn > 0$, then we have $|B_\turn| \geq \beta |B_{\turn-1}| > \lambda^{n+1}\beta$, while if $\turn = 0$ then we have $|B_\turn| \geq 2\rho = \lambda\beta \geq \lambda^{n+1}\beta$. (The length of the ball/interval $B_\turn$ is equal to twice its radius.) So either way we have $|B_\turn| \geq \lambda^{n+1}\beta$. Thus, we have $|I| = \lambda^n \epsilon = \alpha\lambda^{n+1}\beta \leq \alpha|B_\turn|$, so on turn $\turn$ Alice is allowed to delete the interval $I$. Her strategy specifies that she does so, which completes the proof.
\end{proof}

\begin{lemma}
\label{lemmaBA}
For all $0 < \epsilon < 1/2$ and $\big(\frac{\epsilon}{1 - \epsilon}\big)^2 \leq \beta < 1$, the set $\BA_1(\epsilon)$ is $(\alpha,\beta,\rho)$-absolute winning, where
\[
(\alpha,\beta,\rho) = \left(\tfrac{2\epsilon}{1 - 2\epsilon} \beta^{-1},\beta,\beta/2\right).
\]
\end{lemma}
\begin{proof}
For each $p/q\in \Q$ let
\[
\Delta_\epsilon(p/q) = B(p/q,\epsilon q^{-2})
\]
and note that
\[
\BA_1(\epsilon) = \R\butnot\bigcup_{p/q\in\Q} \Delta_\epsilon(p/q).
\]
Now we have
\[
\left|\frac{p_1}{q_1} - \frac{p_2}{q_2}\right| \geq \frac{1}{q_1 q_2}
\]
for all $p_1/q_1 \neq p_2/q_2$. Now fix $\ell > 0$ and let
\[
\CC_\ell := \{\Delta_\epsilon(p/q) : \ell < (1-2\epsilon) q^{-2} \leq \beta^{-1}\ell\}.
\]
If $\Delta_\epsilon(p_1/q_1),\Delta_\epsilon(p_2/q_2)$ are two distinct members of $\CC_\ell$ with $q_1 \leq q_2$, then the distance between them is
\[
\left|\frac{p_1}{q_1} - \frac{p_2}{q_2}\right| - \frac{\epsilon}{q_1^2} - \frac{\epsilon}{q_2^2} \geq \frac{1}{q_1 q_2} - \frac{\epsilon}{q_1^2} - \frac{\epsilon}{q_2^2} = \frac{1}{q_2^2} \left[\frac{q_2}{q_1} - \epsilon\left(\frac{q_2}{q_1}\right)^2 - \epsilon\right] \geq \frac{1}{q_2^2}(1 - 2\epsilon) > \ell,
\]
where the second-to-last inequality is derived from the fact that $x - \epsilon x^2 - \epsilon \geq 1 - 2\epsilon$ for all $x\in [1,1/\epsilon - 1]$ and in particular for all $x\in [1,\beta^{-1/2}]$. Thus each interval of length $\ell$ intersects at most one member of $\CC_\ell$.

Now Alice's strategy can be given as follows: If Bob chooses an interval $B_\turn$ of length $\ell = |B_\turn|$, then Alice deletes the unique member $\Delta_\epsilon(p/q)$ of the collection $\CC_\ell$ that intersects $B_\turn$. By construction, the length of this member is
\[
|\Delta_\epsilon(p/q)| = \frac{2\epsilon}{q^2} \leq \frac{2\epsilon}{1 - 2\epsilon} \beta^{-1} \ell = \alpha |B_\turn|,
\]
meaning that it is legal to play. To show that this strategy is winning, we observe that the length of Bob's first interval $B_0$ is at least $2\rho = \beta$, and thus $\beta^{-1} |B_0| \geq 1 \geq q^{-2}$ for all $p/q\in\Q$. So if $\turn$ is the smallest integer such that $|B_\turn| < q^{-2}$, then $|B_\turn| < q^{-2} \leq \beta^{-1} |B_\turn|$ and thus Alice must delete $\Delta_\epsilon(p/q)$ on turn $\turn$.
\end{proof}

\begin{lemma}
\label{lemmaFn}
For all $n\geq 2$ and $\frac{1}{n^2} \leq \beta < 1$, the set $(-\infty,0)\cup F_n\cup (1,\infty)$ is $(\alpha,\beta,\rho)$-absolute winning, where
\[
(\alpha,\beta,\rho) = \left(\tfrac{2}{n-1} \beta^{-1},\beta,\beta/2\right).
\]
\end{lemma}
\begin{proof}
This follows immediately from Lemma \ref{lemmaBA} together with the inclusion
\[
[0,1]\cap \BA_1(\tfrac{1}{n+1}) \subset F_n
\]
(see e.g. \cite[Theorems 6 and 9]{Khinchin_book}).
\end{proof}

We now use the above lemmas to prove Theorems \ref{theoremM49}, \ref{theoremCF19}, and \ref{theoremF49}.

\begin{proof}[Proof of Theorem \ref{theoremM49}]
We will in fact show that every $a\in M_\epsilon$ is contained in some arithmetic progression of length 3. Indeed, apply Lemma \ref{lemmaMepsilon} with $0 < \epsilon \leq 1/49$ and $\beta = 1/6$ and combine with Proposition \ref{propositionmonotonicity1} to get that $S_1 \df (-\infty,0)\cup M_\epsilon \cup (1,\infty)$ is $(1/4,1/6,1/24)$-absolute winning. By Proposition \ref{propositioninvariance1} it follows that $S_2 \df 2S_1 - a$ is $(1/4,1/6,1/12)$-absolute winning and thus by Proposition \ref{propositionintersection1}, $S_1\cap S_2$ is $(1/4,1/6,1/12,2)$-absolute winning. Since $2(1/4)+3(1/6)=1$, Corollary \ref{corollarynonempty} shows that $S_1\cap S_2 \cap [0,1/6]\neq \emptyset$ and $S_1\cap S_2\cap [5/6,1] \neq \emptyset$; in particular $S_1\cap S_2\cap [0,1]\butnot\{a\} \neq \emptyset$. If $t\in S_1\cap S_2 \cap [0,1]\butnot\{a\}$, then $\{a,\frac{a+t}{2},t\}$ is an arithmetic progression of length 3 on $M_\epsilon$. The proof for $F_{49}$ is similar, using Lemma \ref{lemmaFn} instead of Lemma \ref{lemmaMepsilon}.
\end{proof}

\begin{proof}[Proof of Theorem \ref{theoremCF19}]
Apply Lemma \ref{lemmaFn} with $n=19$ and $\beta = 1/3$ to get that $(-\infty,0)\cup F_{19}\cup (1,\infty)$ is $(1/3,1/3,1/6)$-absolute winning. Now in the $(1/3,1/3,1/6)$-absolute game, Bob can start by playing the interval $[0,1]$ and on each turn can legally play an interval appearing in the construction of the Cantor set (cf. the proof of Corollary \ref{corollarynonempty}). The outcome of the resulting game will be a member of $F_{19} \cap C$, so we have $F_{19} \cap C \neq \emptyset$.
\end{proof}

\begin{proof}[Proof of Theorem \ref{theoremF49}]
Apply Lemma \ref{lemmaFn} with $n=49$ and $\beta = 1/6$ to get that $S_1 \df (-\infty,0)\cup F_{49}\cup (1,\infty)$ is $(1/4,1/6,1/12)$-absolute winning. Now fix $t\in [1/6,11/6]$ and observe that by Proposition \ref{propositioninvariance1} the set $S_2 \df t - S_1$ is also $(1/4,1/6,1/12)$-absolute winning, and thus $S = S_1\cap S_2$ is $(1/4,1/6,1/12,2)$-absolute winning. Since $t\in [1/6,11/6]$, the length of the interval $I \df [\max(0,t-1),\min(1,t)]$ is at least $2\rho = 1/6$. On the other hand, we have $2(1/4) + 3(1/6) = 1$, and thus by Corollary \ref{corollarynonempty}, we have $S \cap I \neq \emptyset$. But $S\cap I \subset F_{49}\cap (t-F_{49})$, so we get $t\in F_{49}+F_{49}$, which completes the proof.
\end{proof}

Although the absolute game is good at getting simple quantitative results when small numbers are involved, the bounds coming from Lemma \ref{lemmanonempty} get worse asymptotically as the Hausdorff dimension of the sets in question tends to 1, and in particular are not good enough to prove Theorem \ref{theoremarithbounds}. To get a better asymptotic estimate we need to introduce another game, the potential game.

\section{The potential game and its properties}
\label{sectionpotential}

We now define the potential game. In the next two sections we will use it to prove Theorems \ref{theoremarithbounds}, \ref{theoremBAdcodim}, and \ref{theoremKepsilondelta}. Like in the previous section, the version of the potential game we give below is slightly different from the original one found in \cite[Appendix C]{FSU4}. The first two changes are the same as for the absolute game (introducing the parameter $\rho$, and splitting $\beta$ into two different parameters), and we also introduce the possibility that the parameter $c$ is equal to zero, to provide a clearer relation between the absolute game and the potential game.

\begin{definition}[Cf. {\cite[Definition C.4]{FSU4}}]
\label{Potential}
Let $X$ be a complete metric space and let $\HH$ be a collection of closed subsets of $X$. Given $\alpha,\beta,\rho > 0$ and $c\geq 0$, Alice and Bob play the \emph{$(\alpha,\beta,c,\rho,\HH)$-potential game} as follows:
\begin{itemize}
\item As before the turn order is alternating, with Bob playing first.
\item On the $\turn$th turn, Bob plays a closed ball $B_\turn = B(x_\turn,\rho_\turn)$, and Alice responds by choosing a finite or countably infinite collection $\AA_\turn$ of sets of the form $\NN(\hyp_{i,\turn},\rho_{i,\turn})$, with $\hyp_{i,\turn}\in\HH$ and $\rho_{i,\turn} > 0$, satisfying
\begin{equation}
\label{alicerules}
\sum_i \rho_{i,\turn}^c \leq (\alpha \rho_\turn)^c,
\end{equation}
where $\NN(\hyp_{i,\turn},\rho_{i,\turn})$ denotes the $\rho_{i,\turn}$-thickening of $\hyp_{i,\turn}$, i.e.
\[
\NN(\hyp_{i,\turn},\rho_{i,\turn}) = \{x\in X : \dist(x,\hyp_{i,\turn}) \leq \rho_{i,\turn}\}.
\]
As before we say that Alice ``deletes'' the collection $\AA_\turn$, though the meaning of this will be slightly different from what it was in the absolute game.

If $c = 0$, then instead of requiring \eqref{alicerules}, we require the collection $\AA_\turn$ to consist of a single element, which must have thickness $\leq \alpha \rho_\turn$.
\item On the first (0th) turn, Bob's ball $B_0 = B(x_0,\rho_0)$ is required to satisfy
\begin{equation}
\label{bobrules1}
\rho_0 \geq \rho,
\end{equation}
while on subsequent turns his ball $B_{\turn + 1} = B(x_{\turn+1},\rho_{\turn+1})$ is required to satisfy
\begin{equation}
\label{bobrules}
\rho_{\turn + 1}\geq \beta \rho_\turn \text{ and } B_{\turn + 1} \subset B_\turn,
\end{equation}
with no reference made to the collection $\AA_\turn$ chosen by Alice on the previous turn.
\end{itemize}
As before the result after infinitely many turns is an infinite descending sequence of balls $B_0 \supset B_1 \supset \cdots$, and if the radii of these balls does not tend to zero we say that Alice wins by default, while otherwise we call the intersection point $x_\infty$ the outcome of the game. However, we now make the additional rule that if the outcome of the game is a member of any ``deleted'' element $\NN(\hyp_{i,\turn},\rho_{i,\turn})$ of one of the collections $\AA_\turn$ chosen throughout the game, then Alice wins by default.

Now let $S \subset X$, and suppose that Alice has a strategy guaranteeing that if she does not win by default, then $x_\infty \in S$. Then the set $S$ is called \emph{$(\alpha,\beta,c,\rho,\HH)$-potential winning}.
\end{definition}

\begin{remark}
\label{remarkc0vsabsolute}
Let $\PP$ be the collection of singletons in $X$. When $c = 0$ and $\HH = \PP$, the potential game is similar to the absolute game considered in the previous section. The only difference is that while in the absolute game Bob must immediately move to avoid Alice's choice, in the potential game he must only do so eventually. This is a significant difference because it means that Bob gets a much larger advantage from having $\alpha$ small, since it means he can wait several turns before avoiding a region.

Thus, every $(\alpha,\beta,0,\rho,\PP)$-potential winning set is $(\alpha,\beta,\rho)$-absolute winning, but the converse is not true. However, the proofs of Lemmas \ref{lemmaMepsilon}-\ref{lemmaFn} in fact show that the sets in question are $(\alpha,\beta,0,\rho,\PP)$-potential winning, since the proofs only use the fact that the outcome is not in any deleted set, not that the deleted sets are disjoint from Bob's subsequent moves. This fact will be used in the applications below.
\end{remark}

\begin{notation}
\label{notationthi}
In what follows we will use the notation
\[
\thi(\NN(\hyp,\rho)) = \rho,
\]
i.e. $\thi(\NN(\hyp,\rho))$ is the ``thickness'' of $\NN(\hyp,\rho)$.
\end{notation}


The basic properties of the potential game are similar to those for the absolute game. We omit the proofs as they are essentially the same as the proofs of Propositions \ref{propositionintersection1}-\ref{propositioninvariance1}.

\begin{proposition}[Countable intersection property]
\label{propositionintersection2}
Let $J$ be a countable (finite or infinite) index set, and for each $j\in J$, let $S_j$ be an $(\alpha_j,\beta,c,\rho,\HH)$-potential winning set, where $c > 0$. 
Then the set $S = \bigcap_{j\in J} S_j$ is $(\alpha,\beta,c,\rho,\HH)$-potential winning, where
\begin{equation}
\label{alphacsum}
\alpha^c = \sum_{j\in J} \alpha_j^c,
\end{equation}
assuming that the series converges.
\end{proposition}

\begin{proposition}[Monotonicity]
\label{propositionmonotonicity2}
If $S$ is $(\alpha,\beta,c,\rho,\HH)$-potential winning and $\alpha\leq\w\alpha$, $\beta \leq \w\beta$, $c\leq \w c$, $\rho \leq \w\rho$, and $\HH \subset \w\HH$, then $S$ is $(\w\alpha,\w\beta,\w c,\w\rho,\w\HH)$-potential winning.
\end{proposition}

Note that the proof of Proposition \ref{propositionmonotonicity2} uses the H\"older inequality
\[
\left(\sum_i \alpha_i^{\w c}\right)^{1/\w c} \leq \left(\sum_i \alpha_i^c\right)^{1/c} \text{ when } c \leq \w c.
\]

\begin{proposition}[Invariance under similarities]
\label{propositioninvariance2}
Let $f:X\to Y$ be a bijection that satisfies
\[
\dist(f(x),f(y)) = \lambda \dist(x,y) \all x,y\in X.
\]
Then a set $S \subset X$ is $(\alpha,\beta,c,\rho,\HH)$-absolute winning if and only if the set $f(S) \subset Y$ is $(\alpha,\beta,c,\lambda\rho,f(\HH))$-absolute winning.
\end{proposition}

\ignore{

\begin{proposition}[Invariance under bi-Lipschitz maps]
\label{propositionlipinv}
If $S$ is $(\alpha,\beta,c,\rho,\HH)$-potential winning and $f:X\to Y$ is a $K_3$-Lipschitz homeomorphism whose inverse is $K_4$-Lipschitz, then $f(S) \subset Y$ is $(K_3 K_4\alpha,\beta,c,K_3\rho,f(\HH))$-potential winning.
\end{proposition}
\begin{proof}
Let us call the $(\alpha,\beta,c,\rho,\HH)$-potential game Game I and the $(K_3 K_4\alpha,\beta,c,K_3\rho,f(\HH))$-potential game Game II. We need to show that Alice can transfer any winning strategy from Game I to Game II. For each move $\w B_\turn = B(x,r)$ Bob makes in Game II, Alice pretends that he has made the corresponding move $B_\turn = B(f^{-1}(x),K_4 r)$ in Game I. The inclusions $\w B_{\turn + 1} \subset \w B_\turn$ in Game II together with the fact that $f^{-1}$ is $K_4$-Lipschitz imply that the corresponding inclusions $B_{\turn + 1} \subset B_\turn$ hold in Game I, i.e. Bob is playing legally in Game I. In Game I Alice plays her winning strategy, which means that in response to each of Bob's balls $B_\turn$, she deletes some collection $\AA_\turn$. She then transfers her strategy to Game II by deleting the corresponding collection
\[
\w\AA_\turn = \{\NN(f(\hyp),K_3 r) : \NN(\hyp,r) \in \AA_\turn\}.
\]
This is legal to play since
\[
\sum_{A\in\w\AA_\turn} \thi^c(A) = \sum_{A\in\AA_\turn} (K_3 \thi(A))^c \leq (K_3 \alpha \rad(B_\turn))^c = (K_3 K_4 \alpha \rad(\w B_\turn))^c.
\]
Since $f$ is $K_3$-Lipschitz, we have $f(\bigcup[\AA_\turn]) \subset \bigcup\w\AA_\turn$, and thus if Alice wins by default in Game I, then she wins by default in Game II as well. This completes the proof.
\end{proof}


\ignore{
A map is called \emph{$(C,\theta)$-quasisymmetric} if for all $x,y,z\in X$ such that $\dist(x,z) \geq \dist(x,y)$, we have
\[
\frac{\dist(fx,fz)}{\dist(fx,fy)} \leq C\left(\frac{\dist(x,z)}{\dist(x,y)}\right)^\theta.
\]

\begin{proposition}
Suppose $X$ is geodesic. If $S$ is $(\alpha,\beta,c,\rho,\HH)$-potential winning and $f:X\to Y$ is a $(C_1,\theta_1)$-quasisymmetric homeomorphism whose inverse is $(C_2,\theta_2)$-quasisymmetric, then $f(S) \subset Y$ is $(C^2\alpha^\theta,\beta^\theta,c/\theta,?,f(\HH)$-potential winning.
\end{proposition}
\begin{proof}
Let us call the $(\alpha,\beta,c,\rho,\HH)$-potential game Game I and the $(C^2\alpha^\theta,\beta^\theta,c/\theta,?,f(\HH)$-potential game Game II. We need to show that Alice can transfer any winning strategy from Game I to Game II. For each move $\w B_\turn = B(x,r)$ Bob makes in Game II, Alice pretends that he has made the corresponding move $B_\turn = B(f^{-1}(x),r')$ in Game I, where $r'$ is chosen so that $f(B_\turn) \supset \w B_\turn$, and is minimal subject to this condition. [How to verify $B_{\turn+1} \subset B_\turn$?\internal] The inclusions $\w B_{\turn + 1} \subset \w B_\turn$ in Game II together with the fact that $f^{-1}$ is $K_4$-Lipschitz imply that the corresponding inclusions $B_{\turn + 1} \subset B_\turn$ hold in Game I, i.e. Bob is playing legally in Game I. In Game I Alice plays her winning strategy, which means that in response to each of Bob's balls $B_\turn$, she deletes some collection $\AA_\turn$. She then transfers her strategy to Game II by deleting the corresponding collection
\[
\w\AA_\turn = \{\NN(f(\hyp),K_3 r) : \NN(\hyp,r) \in \AA_\turn\}.
\]
This is legal to play since
\[
\sum_{A\in\w\AA_\turn} \thi^c(A) = \sum_{A\in\AA_\turn} (K_3 \thi(A))^c \leq (K_3 \alpha \rad(B_\turn))^c = (K_3 K_4 \alpha \rad(\w B_\turn))^c.
\]
Since $f$ is $K_3$-Lipschitz, we have $f(\bigcup[\AA_\turn]) \subset \bigcup\w\AA_\turn$, and thus if Alice wins by default in Game I, then she wins by default in Game II as well. This completes the proof.
\end{proof}

}

\begin{proposition}[$\rho$-comparison lemma]
\label{propositionrhocomparison}
If $S$ is $(\alpha,\beta,c,\rho,\HH)$-potential winning and $\lambda \geq 1$ then $S$ is $(\lambda\alpha,\beta,c,\lambda^{-1}\rho,\HH)$-potential winning.
\end{proposition}
\begin{proof}
Let us call the $(\alpha,\beta,c,\rho,\HH)$-potential game Game I and the $(\lambda\alpha,\beta,c,\lambda^{-1}\rho,\HH)$-potential game Game II. We need to show that Alice can transfer any winning strategy from Game I to Game II. She transfers her strategy as follows: every time Bob makes a move $B_\turn = B(x,\rho_\turn)$ in Game II, she pretends that Bob has made the move $\w B_\turn = B(x,\lambda \rho_\turn)$ in Game I. Since $\lambda \geq 1$, the inclusions $B_{\turn+1} \subset B_\turn$ in Game II imply that the inclusions $\w B_{\turn + 1} \subset \w B_\turn$ hold for Game I, i.e. Bob is playing legally in Game I. In Game I Alice plays her winning strategy, which means that in response to each of Bob's balls $\w B_\turn$, she deletes some collection $\AA_\turn$. She then transfers her strategy by deleting the exact same collection in response to the corresponding move $B_\turn$ in Game II. This is legal because $(\lambda\alpha)\rho_\turn = \alpha (\lambda \rho_\turn)$.
\end{proof}

\begin{proposition}[$\beta$-comparison lemma]
\label{propositionbetacomparison}
Suppose that $X$ is geodesic. If $S$ is $(\alpha,\beta,c,\rho,\HH)$-potential winning and $m\in\N$ then $S$ is $((1+\beta^{-c}+\cdots+\beta^{-(m-1)c})^{1/c}\alpha,\beta^m,c,\beta^{m-1}\rho,\HH)$-potential winning. [$m\neq \turn$\internal]
\end{proposition}
\begin{proof}
The proof is similar to the proof of the previous proposition. Let us call the $(\alpha,\beta,c,\rho,\HH)$-potential game Game I and the $((1+\beta^{-c}+\cdots+\beta^{-(m-1)c})^{1/c}\alpha,\beta^m,c,\rho,\HH)$-potential game Game II. Alice can transfer a winning strategy from Game I to Game II as follows:  every time Bob makes a move $B_\turn = B(x,\rho_\turn)$ in Game II, she pretends that Bob has made a sequence of moves $\w B_{mn},\ldots,\w B_{mn+m-1}$ satisfying
\[
B_{\turn-1} = \w B_{mn-1} \supset \w B_{mn} \supset \cdots \supset \w B_{mn+m-1} = B_\turn
\]
as well as the inequalities $\rad(\w B_{k+1}) \geq \beta\,\rad(\w B_k)$. Such a choice is possible due to the fact that $X$ is geodesic, as well as to the inequality $\rad(B_\turn) \geq \beta^m \,\rad(B_{\turn-1})$ and the inclusion $B_\turn \subset B_{\turn-1}$. If $n = 0$, we omit the first inclusion and radial inequality but instead require that $\rad(\w B_0) \geq \rho$. Finally, Alice uses her winning strategy for Game I against the moves $\w B_{mn-m+1},\ldots,\w B_{mn}$, and combines the resulting collections to form a single collection which she plays in response to $B_\turn$. It can be checked that the choice of $\w\alpha = (1+\beta^{-c}+\cdots+\beta^{-(m-1)c})^{1/c}\alpha$ means that this strategy is legal.
\end{proof}

}

\section{Hausdorff dimension of potential winning sets}
\label{sectionpotentialHD}

In this section we will address the question: what is the appropriate analogue of Lemma \ref{lemmanonempty} for the potential game? In other words, what quantitative information about Hausdorff dimension can be deduced from the assumption that a certain set is potential winning? To answer this question, we first need some definitions.

\begin{definition}
Given $\delta > 0$, a measure $\mu$ on a complete metric space $X$ is said to be \emph{Ahlfors $\delta$-regular} if for every sufficiently small ball $B(x,\rho)$ centered in the topological support of $\mu$, we have $\mu(B(x,\rho)) \asymp \rho^\delta$. The topological support of an Ahlfors $\delta$-regular measure is also said to be Ahlfors $\delta$-regular.

Given $\eta > 0$ and a collection of closed sets $\HH$ in $X$, the measure $\mu$ is called \emph{absolutely $(\eta,\HH)$-decaying} if for every sufficiently small ball $B(x,\rho)$ centered in the topological support of $\mu$, for every $\hyp\in \HH$, and for every $\epsilon > 0$, we have
\[
\mu(B(x,\rho)\cap \NN(\hyp,\epsilon\rho)) \lesssim \epsilon^\eta \mu(B(x,\rho)).
\]
When $X = \R^d$ and $\HH$ is the collection of hyperplanes, then $\mu$ is called \emph{absolutely $\eta$-decaying}.

Finally, the \emph{Ahlfors dimension} of a (not necessarily closed) set $S \subset \R$ is the supremum of $\delta$ such that $S$ contains a closed Ahlfors $\delta$-regular subset. We will denote it by $\AD(S)$. The Ahlfors dimension of a set is a lower bound for its Hausdorff dimension.
\end{definition}

\begin{example}
\label{exampleabspoints}
Every Ahlfors $\delta$-regular measure is absolutely $(\delta,\PP)$-decaying, where $\PP$ is the collection of singletons in $X$.
\end{example}

\begin{example}
\label{example1decay}
Lebesgue measure on $\R^d$ is absolutely $1$-decaying.
\end{example}

The following theorem is a combination of known results:
\begin{theorem}
\label{theoremknown}
Let $X = \R^d$, let $\HH$ be the collection of hyperplanes, and let $J$ be the support of an Ahlfors $\delta$-regular and absolutely $(\eta,\HH)$-decaying measure $\mu$. Suppose that $S \subset X$ is $(\alpha,\beta,c,\rho,\HH)$-potential winning for all $\alpha,\beta,c,\rho > 0$. Then we have $\AD(S\cap J) = \delta$.
\end{theorem}
\begin{proof}
The set $S$ is $\HH$-potential winning in the terminology of \cite[Appendix C]{FSU4}, and thus by \cite[Theorem C.8]{FSU4} it is also $\HH$-absolute winning, or in other words hyperplane absolute winning. So by \cite[Propositions 4.7 and 5.1]{BFKRW} $S$ is winning on $J$, and thus by \cite[Proposition D.1]{FSU4} we have $\AD(S\cap J) = \delta$.
\end{proof}

In this paper we will be interested in the following quantitative version of Theorem \ref{theoremknown}:

\begin{theorem}
\label{theorempotentialHD}
Let $X$ be a complete metric space, $\HH$ a collection of closed subsets of $X$, and $J\subset X$ be the topological support of an Ahlfors $\delta$-regular and absolutely $(\eta,\HH)$-decaying measure $\mu$. Let $S\subset X$ be $(\alpha,\beta,c,\rho,\HH)$-potential winning, with $c < \eta$ and $\beta \leq 1/4$. Then for every ball $B_0\subset X$ centered in $J$ with $\rad(B_0) \geq \rho$, we have
\begin{equation}
\label{potentialHD}
\AD(S\cap J\cap B_0) \geq \delta - K_1 \frac{\alpha^\eta}{|\log(\beta)|} > 0 \;\;\;\text{ if }\;\;\; \alpha^c \leq \frac{1}{K_2}(1-\beta^{\eta-c})
\end{equation}
where $K_1,K_2$ are large constants independent of $\alpha,\beta,c,\rho$ (but possibly depending on $X,J,\HH$).
\end{theorem}

\begin{remark}
Theorem \ref{theorempotentialHD} implies that Theorem \ref{theoremknown} is true for every complete metric space $X$ and every collection $\HH$ of closed subsets of $X$. In particular, the condition \cite[Assumption C.6]{FSU4}, which is crucial for establishing \cite[Theorem C.8]{FSU4}, turns out not to be necessary for proving its consequences in terms of Hausdorff dimension.
\end{remark}

\begin{proof}[Proof of Theorem \ref{theorempotentialHD}]
For each $\index\geq 0$ let $\rho_\index = \beta^\index \rho$, let $\sep_\index \subset J$ be a maximal $\rho_\index/2$-separated subset, and let
\[
\EE_\index = \{B(x,\rho_\index) : x\in \sep_\index\}.
\]
Let $\pi_\index:\EE_{\index + 1} \to \EE_\index$ be a map such that for all $B\in \EE_{\index + 1}$, we have
\begin{equation}
\label{pindef}
B \subset \pi_\index(B).
\end{equation}
Such a map exists since $\beta \leq 1/2$. (We will later impose a further restriction on the map $\pi_\index$.) When $m < n$ and $B\in \EE_n$, we will abuse notation slightly by writing $\pi_m(B) = \pi_m\circ\pi_{m+1}\circ\cdots\circ \pi_{n-1}(B)$.

For each $B\in \EE_\index$, consider the sequence of moves in the $(\alpha,\beta,c,\rho,\HH)$-potential game where for each $\turn = 0,\ldots,\index$, on the $\turn$th turn Bob plays the move $\pi_\turn(B)$, and Alice responds according to her winning strategy. By \eqref{pindef}, Bob's moves are all legal. Let $\AA(B)$ denote Alice's response on turn $\index$ according to her winning strategy. Also, let $\AA_\turn^*(B) = \{A\in \AA(\pi_\turn(B)) : B\cap A \neq \emptyset\}$.

\ignore{
Finally, let
\[
\AA_\index = \bigcup_{B\in B_\index} \{B\cap A : A\in \AA(B)\}.
\]
For each $A\in \AA_\index$, we let $\thi(A)$ denote the thickness of $A$, i.e. $\thi(B(x,\ar)\cap \thickvar{A}{\ar'}) = \ar'$.
\begin{claim}
For all $x\in X\butnot \bigcup_{\index\in\N} \bigcup[\AA_\index]$, we have $x\in S$.
\end{claim}
\begin{proof}
By K\doubleacute onig's lemma, there exists a sequence $\EE_\index \ni B_\index \to x$ such that for all $\index$, we have $\pi(B_{\index + 1}) = B_\index$. Consider the strategy for Bob consisting of the plays $B_0,B_1,\ldots$. Alice's responses are the sets $\AA(B_\index)$ ($n\in\N$). Since Alice's strategy is winning, we have
\[
x\in S \cup \bigcup_{\index\in\N} \bigcup[\AA(B_\index)].
\]
On the other hand, $x\in \bigcap_{\index\in\N} B_\index$. So either $x\in S$, or there exists $n\in\N$ such that $x\in B_\index\cap \bigcup[\AA(B_\index)]$. In the former case we are done, and in the latter case we have $x\in \bigcup_{\index\in\N} \bigcup[\AA_\index]$, a contradiction.
\end{proof}

}


Fix $\epsilon > 0$ small to be determined, independent of $\alpha,\beta,c,\rho$, and let
\begin{equation}
\label{Ndef}
N = \lfloor \epsilon\alpha^{-\eta}\rfloor.
\end{equation}
For each $j\geq 0$ let $D_j \subset \sep_{jN}$ be a maximal $3\rho_{jN}$-separated set, and let $\DD_j = \{B(x,\rho_{jN}) : x\in D_j\} \subset \EE_{jN}$. Note that $\DD_j$ is a disjoint collection. For each $B\in \DD_j$ let
\[
\phi_j(B) = \sum_{\index < jN} \sum_{A\in \AA_\index^*(B)} \thi^c(A)
\]
(cf. Notation \ref{notationthi}). Fix $\gamma > 0$ small to be determined, independent of $\alpha,\beta,c,\rho$, and let
\[
\DD_j' = \{B \in \DD_j : \phi_j(B) \leq (\gamma \rho_{jN})^c\}.
\]
For every ball $B$ let
\[
\DD_{j+1}(B) = \big\{B' \in \DD_{j + 1} : B' \subset \tfrac12 B\big\},
\]
where $\lambda B$ denotes the ball resulting from multiplying the radius of $B$ by $\lambda$ while leaving the center fixed.
\begin{claim}
\label{claimdescendants}
For all $B\in \DD_j'$, we have
\begin{equation}
\label{kept-balls-count}
\#\big(\DD_{j+1}(B)\cap \DD_{j+1}'\big) \gtrsim \beta^{-N\delta} \;\;\;\text{ if }\;\;\; \alpha^c \leq \frac{1}{K_2}(1-\beta^{\eta-c}),
\end{equation}
where $K_2$ is a large constant.
\end{claim}
\begin{proof}
The Ahlfors regularity of $J$ implies that the cardinality of $\DD_{j+1}(B)$ is at least $\frac{1}{K_3} \beta^{-N\delta}$, where $K_3$ is a large constant. Thus we just need to show that
\begin{equation}
\label{ETSdescendants}
\#(\DD_{j+1}(B)\butnot \DD_{j+1}') \leq  \frac{1}{2K_3} \beta^{-N\delta}.
\end{equation}
Now
\begin{align*}
\#(\DD_{j+1}(B)\butnot \DD_{j+1}')
&\leq \sum_{B'\in \DD_{j+1}(B)} \min\left(1,\frac{\phi_{j+1}(B')}{(\gamma \rho_{(j+1)N})^c}\right)\\
&\leq \sum_{B'\in \DD_{j+1}(B)} \sum_{\index < (j+1)N} \sum_{A\in \AA_\index^*(B')} \min\left(1,\frac{\thi^c(A)}{(\gamma \rho_{(j+1)N})^c}\right)\\
&\leq \sum_{\index < jN} \sum_{A\in \AA_\index^*(B)} \min\left(1,\frac{\thi^c(A)}{(\gamma \rho_{(j+1)N})^c}\right)\#\{B'\in \DD_{j+1}(B) : B'\cap A\neq \emptyset\}\\
&+ \sum_{jN \leq \index < (j+1)N} \sum_{\substack{B'\in \EE_\index \\ B' \subset B}} \sum_{A\in \AA(B')} \min\left(1,\frac{\thi^c(A)}{(\gamma \rho_{(j+1)N})^c}\right) \#\{B''\in \DD_{j+1}(B') : B''\cap A \neq \emptyset\}.
\end{align*}
The idea is to bound the first term (representing ``old'' obstacles) using the assumption that $B\in \DD_j'$, which implies that $\phi_j(B) \leq (\gamma\rho_{jN})^c$, and to bound the second term (representing ``new'' obstacles) using the fact that Alice is playing legally, which implies \eqref{alicerules}. To do this, we observe that for all $B' \in \bigcup_\index \EE_\index$ and $A = \NN(\LL,\thi(A))$, since $\mu$ is Ahlfors $\delta$-regular and absolutely $(\eta,\HH)$-decaying we have
\begin{align*}
\#\{B'' \in \DD_{j+1}(B') : B'' \cap A \neq \emptyset\}
&\lesssim \frac{1}{\rho_{(j+1)N}^\delta} \mu\big(B'\cap \thickvar A{2\rho_{(j+1)N}}\big)\\
&\lesssim \frac{1}{\rho_{(j+1)N}^\delta} \left(\frac{\thi(\thickvar A{2\rho_{(j+1)N}})}{\rad(B')}\right)^\eta \,\rad^\delta(B')\\
&= \left(\frac{\rad(B')}{\rho_{(j+1)N}}\right)^\delta \left(\frac{\thi(A) + 2\rho_{(j+1)N}}{\rad(B')}\right)^\eta
\end{align*}
and thus
\begin{align*}
\#(\DD_{j+1}(B)\butnot \DD_{j+1}')
&\lesssim \left(\frac{\rad(B)}{\rho_{(j+1)N}}\right)^\delta \sum_{\index < jN} \sum_{A\in \AA_\index^*(B)} \min\left(1,\frac{\thi^c(A)}{(\gamma \rho_{(j+1)N})^c}\right) \left(\frac{\thi(A) + 2\rho_{(j+1)N}}{\rad(B)}\right)^\eta\\
&+ \sum_{jN \leq \index < (j+1)N} \sum_{\substack{B'\in \EE_\index \\ B' \subset B}} \left(\frac{\rad(B')}{\rho_{(j+1)N}}\right)^\delta \sum_{A\in \AA(B')} \min\left(1,\frac{\thi^c(A)}{(\gamma \rho_{(j+1)N})^c}\right)  \left(\frac{\thi(A) + 2\rho_{(j+1)N}}{\rad(B')}\right)^\eta.
\end{align*}
To bound this expression, we first prove the following.
\begin{subclaim}
\label{subclaim3etabound}
We have
\begin{equation}
\label{3etabound2}
\sum_{\index < jN} \sum_{A\in \AA_\index^*(B)} \min\left(1,\frac{\thi^c(A)}{(\gamma \rho_{(j+1)N})^c}\right) \left(\frac{\thi(A) + 2\rho_{(j+1)N}}{\rad(B)}\right)^\eta \leq 3^\eta \gamma^c \max\left(\gamma^{\eta-c},\frac{1}{\gamma^c}\left(\frac{\rho_{(j+1)N}}{\rad(B)}\right)^{\eta-c}\right),
\end{equation}
and for all $B'\in \bigcup_\index \EE_\index$, we have
\begin{equation}
\label{3etabound1}
\sum_{A\in \AA(B')} \min\left(1,\frac{\thi^c(A)}{(\gamma \rho_{(j+1)N})^c}\right) \left(\frac{\thi(A) + 2\rho_{(j+1)N}}{\rad(B')}\right)^\eta \leq 3^\eta \alpha^c \max\left(\alpha^{\eta-c},\frac{1}{\gamma^c}\left(\frac{\rho_{(j+1)N}}{\rad(B')}\right)^{\eta-c}\right).
\end{equation}
\end{subclaim}
\begin{remark*}
The proof of this subclaim will show that the left-hand sides of the maxima correspond to the contributions from ``big'' obstacles while the right-hand sides of the maxima correspond to contributions from ``small'' obstacles.
\end{remark*}
\begin{proof}
Let us prove \eqref{3etabound1} first. Since Alice is playing legally, we have
\begin{equation}
\label{potentialbound}
\sum_{A\in \AA(B')} \left(\frac{\thi(A)}{\rad(B')}\right)^c \leq \alpha^c
\end{equation}
so the trick is relating the left-hand side of \eqref{3etabound1} to the left-hand side of \eqref{potentialbound}.

Now, it can be verified that the inequality
\begin{equation}
\label{alphagammaineq}
\min\left(1,\frac{x^c}{(\gamma y)^c}\right) (x + 2y)^\eta \leq 3^\eta x^c \max\left(x^{\eta - c},\frac{y^{\eta - c}}{\gamma^c}\right)
\end{equation}
holds for all $x,y > 0$, e.g. by splitting into the cases $x\geq y$ (use the left option of both ``min'' and ``max'') and $x\leq y$ (use the right option of both ``min'' and ``max''). Letting $x = \thi(A)/\rad(B')$ and $y = \rho_{(j+1)N}/\rad(B')$ and summing over all $A\in \AA(B')$ shows that
\begin{align*}
&\sum_{A\in \AA(B')} \min\left(1,\frac{\thi^c(A)}{(\gamma \rho_{(j+1)N})^c}\right) \left(\frac{\thi(A) + 2\rho_{(j+1)N}}{\rad(B')}\right)^\eta\\
&\leq \sum_{A\in \AA(B')} 3^\eta \left(\frac{\thi(A)}{\rad(B')}\right)^c\max\left(\left(\frac{\thi(A)}{\rad(B')}\right)^{\eta-c},\frac{1}{\gamma^c}\left(\frac{\rho_{(j+1)N}}{\rad(B')}\right)^{\eta-c}\right)\\
&\leq 3^\eta \left(\sum_{A\in \AA(B')} \left(\frac{\thi(A)}{\rad(B')}\right)^c\right) \max\left(\left(\max_{A\in \AA(B')}\frac{\thi(A)}{\rad(B')}\right)^{\eta-c},\frac{1}{\gamma^c}\left(\frac{\rho_{(j+1)N}}{\rad(B')}\right)^{\eta-c}\right).
\end{align*}
Applying \eqref{potentialbound} twice yields \eqref{3etabound1}.

The proof of \eqref{3etabound2} is similar, except that instead of summing over $A\in \AA(B')$, we sum over  $A \in \bigcup_{\index < jN} \AA_\index^*(B)$, and instead of \eqref{potentialbound}, we use the fact that the assumption $B\in \DD_j'$ implies that
\[
\sum_{\index < jN} \sum_{A\in \AA_\index^*(B)} \left(\frac{\thi(A)}{\rad(B)}\right)^c \leq \gamma^c.
\]
This completes the proof of Subclaim \ref{subclaim3etabound}.
\end{proof}
\vspace{0.1 in}
\noindent Combining Subclaim \ref{subclaim3etabound} with the inequality preceding it yields
\begin{align*}
\#(\DD_{j+1}(B)\butnot \DD_{j+1}')
&\lesssim \left(\frac{\rad(B)}{\rho_{(j+1)N}}\right)^\delta \gamma^c \max\left(\gamma^{\eta-c},\frac{1}{\gamma^c}\left(\frac{\rho_{(j+1)N}}{\rad(B)}\right)^{\eta-c}\right)\\
&+ \sum_{jN \leq \index < (j+1)N} \sum_{\substack{B'\in \EE_\index \\ B' \subset B}} \left(\frac{\rad(B')}{\rho_{(j+1)N}}\right)^\delta \alpha^c \max\left(\alpha^{\eta-c},\frac{1}{\gamma^c}\left(\frac{\rho_{(j+1)N}}{\rad(B')}\right)^{\eta-c}\right).
\end{align*}
Now by definition we have $\rad(B) = \beta^{jN}\rho$, $\rho_{(j+1)N} = \beta^{(j+1)N}\rho$, and $\rad(B') = \beta^\index\rho$ for all $B'\in \EE_\index$. Thus after applying the change of variables $\index = (j+1)N - k$, we get
\begin{align*}
\frac{\rad(B)}{\rho_{(j+1)N}} &= \beta^{-N},&
\frac{\rad(B')}{\rho_{(j+1)N}} &= \beta^{-k}.
\end{align*}
On the other hand, the Ahlfors regularity of $J$ implies that
\[
\#\{B'\in \EE_\index : B' \subset B\} \asymp \left(\frac{\rad(B)}{\beta^\index \rho}\right)^\delta = \beta^{-(N-k)\delta},
\]
so we have
\begin{equation}
\label{almostenoughclaim}
\#(\DD_{j+1}(B)\butnot \DD_{j+1}')
\lesssim \beta^{-N\delta} \gamma^c \max\left(\gamma^{\eta-c},\frac{1}{\gamma^c}\beta^{N(\eta-c)}\right)
+ \beta^{-N\delta} \alpha^c \sum_{k = 1}^N \max\left(\alpha^{\eta-c},\frac{1}{\gamma^c}\beta^{k(\eta-c)}\right).
\end{equation}
Denote the implied constant of this inequality by $K_4$, and let
\[
\epsilon = \frac1{6 K_3 K_4}\cdot
\]
Then to deduce \eqref{ETSdescendants} from \eqref{almostenoughclaim}, it suffices to show that all four contributions to the right-hand side of \eqref{almostenoughclaim} are less than $\beta^{-N\delta}\epsilon$, i.e. that
\begin{align} \label{sts1}
\gamma^\eta &\leq \epsilon\phantom{.} \hspace{1.5 in} \text{(old big obstacles)}\hspace{-1.5 in}\\ \label{sts2}
\beta^{N(\eta - c)} &\leq \epsilon\phantom{.} \hspace{1.5 in} \text{(old small obstacles)}\hspace{-1.5 in}\\[2pt] \label{sts3}
N \alpha^\eta &\leq \epsilon\phantom{.} \hspace{1.5 in} \text{(new big obstacles)}\hspace{-1.5 in}\\[-6pt] \label{sts4}
\frac{\alpha^c}{\gamma^c} \sum_{k = 0}^\infty \beta^{k(\eta - c)} &\leq \epsilon. \hspace{1.5 in} \text{(new small obstacles)}\hspace{-1.5 in}
\end{align}
Now \eqref{sts1} can be achieved by choosing $\gamma = \epsilon^{1/\eta}$, while \eqref{sts3} is true by the definition of $N$ (see \eqref{Ndef}). This leaves \eqref{sts2} and \eqref{sts4}, which can be rearranged as
\begin{align*}
N(\eta - c)|\log(\beta)| &\geq |\log(\epsilon)|\\
\alpha^c \frac1{1 - \beta^{\eta - c}} &\leq \epsilon \gamma^c = \epsilon^{1 + c/\eta}.
\end{align*}
Now fix $K_2$ large to be determined, and suppose that $\alpha^c \leq \frac{1}{K_2} (1 - \beta^{\eta - c})$. Since $1+c/\eta < 2$, if $K_2 \geq \epsilon^{-2}$ then \eqref{sts4} holds. Moreover, since $\alpha^c \leq \epsilon^2 \leq \epsilon^{c/\eta}$, we have $\epsilon \alpha^{-\eta} \geq 1$ and thus
\[
N = \lfloor \epsilon \alpha^{-\eta}\rfloor \geq \tfrac12 \epsilon\alpha^{-\eta}.
\]
On the other hand, we have
\[
\alpha^\eta \leq \alpha^c \leq \frac{1}{K_2} \big|\log\beta^{\eta - c}\big|
\]
and thus if $K_2 \geq 2\epsilon^{-1}\log(\epsilon^{-1})$, then
\[
\tfrac12\epsilon\alpha^{-\eta} (\eta - c)|\log(\beta)| \geq |\log(\epsilon)|,
\]
demonstrating \eqref{sts2}. So we let
\[
K_2 = \max\big(\epsilon^{-2},2\epsilon^{-1}\log(\epsilon^{-1})\big).
\]
This completes the proof of Claim \ref{claimdescendants}.
\end{proof}

\vspace{0.1 in}
\noindent Let $K_3$ be as in the proof of the claim, so that
\begin{equation}
\label{Mdef}
\#\big(\DD_{j+1}(B)\cap \DD_{j+1}'\big) \geq M \df \left\lceil \frac1{2K_3} \beta^{-N\delta} \right\rceil \text{ for all }B'\in \DD_j'.
\end{equation}
Let $B_0$ be the ball given in the statement of the theorem, and assume that $B_0 \in \DD_0$. (It is always possible to select $\EE_0$ and $\DD_0$ such that this is the case.) Since $\phi_0(B_0) = 0 < (\gamma\rho)^c$, we have $B_0 \in \DD_0'$.

We can now construct a Cantor set $F$ as follows: let $\BB_0 = \{B_0\} \subset \DD_0'$, and whenever we are given a collection $\BB_j \subset \DD_j'$, construct a new collection $\BB_{j+1}$ by replacing each element $B\in \BB_j$ by $M$ elements of $\DD_{j+1}(B)\cap \DD_{j+1}'$. Such elements exist by \eqref{Mdef}. Finally, let
\[
F = \bigcap_{j = 0}^\infty \bigcup_{B\in \BB_j} B.
\]
Then standard arguments (see e.g. \cite{Beardon4}) show that $F$ is Ahlfors regular of dimension
\[
\frac{\log(M)}{|\log(\beta^N)|} \geq \frac{\log(\tfrac1{2K_3}\beta^{-N\delta})}{|\log(\beta^N)|}
= \delta - \frac{\log(2K_3)}{|\log(\beta^N)|}
= \delta - \frac{\log(2K_3)}{N|\log(\beta)|}
\geq \delta - 2\epsilon^{-1}\log(2K_3) \frac{\alpha^\eta}{|\log(\beta)|}\cdot
\]
So to demonstrate the first half of \eqref{potentialHD}, we just need to show that $F \subset S\cap J\cap B_0$. It is clear that $F \subset J\cap B_0$, so we show that $F \subset S$. Indeed, fix $x\in F$. For each $j\in\N$, let $B_{jN}$ be the unique element of $\BB_j$ containing $x$. At this point, we introduce the requirement that for each $j$, the map $\pi_{jN}$ must satisfy
\[
\pi_{jN}(B') = B \text{ whenever } \EE_{jN+1} \ni B' \subset B \in \DD_j.
\]
Due to the disjointness of the collection $\DD_j$, it is possible to choose a map $\pi_{jN}$ satisfying this requirement. Since $\beta \leq 1/4$, if $B\in \DD_j$ and $B'\in \EE_{jN+1}$ satisfy $B'\cap \frac12 B \neq \emptyset$, then $B' \subset B$. It follows that
\[
\pi_{jN}(B') = B \text{ whenever } \EE_n \ni B' \subset \tfrac12 B, \; B \in \DD_j,\; \index > jN.
\]
By the definition of $\BB_{j+1}$ we have $B_{(j+1)N} \subset \tfrac12 B_{jN}$ and thus $\pi_{jN}(B_{(j+1)N}) = B_{jN}$. Thus the partial sequence $(B_n)_{n\in jN\N}$ can be uniquely extended to a full sequence $(B_n)_{n\in\N}$ by requiring that $B_n = \pi_n(B_{n+1})$ for all $n$.

Now interpret the sequence $(B_n)_{n\in\N}$ as a sequence of moves for Bob in the potential game, and suppose Alice responds by playing her winning strategy. Then the outcome of the game is $x$, so either $x\in S$ or Alice wins by default. Suppose that Alice wins by default. Then we have $x\in A\in \AA(B_m)$ for some $m$. It follows that $A \in \AA_m^*(B_n)$ for all $n > m$, and thus
\[
\phi_j(B_{jN}) \geq \thi^c(A)
\]
for all $j$ such that $jN > m$. On the other hand, since $B_{jN} \in \DD_j'$ we have $\phi_j(B_{jN}) \leq (\gamma\rho_{jN})^c$, and thus $\thi(A) \leq \gamma \rho_{jN}$ for all $j$ such that $jN > m$. Letting $j\to\infty$ we get $\thi(A) = 0$, a contradiction. Thus $x\in S$, and hence $F \subset S$. This demonstrates the first half of \eqref{potentialHD}.

To demonstrate the second half of \eqref{potentialHD}, we observe that if $\alpha^c \leq \frac1{K_2}(1-\beta^{\eta-c})$, then
\[
\frac{\alpha^\eta}{|\log(\beta)|} \leq \frac{\alpha^c}{|\log(\beta)|} \leq \frac{1}{K_2}\cdot\frac{|\log(\beta^{\eta-c})|}{|\log(\beta)|} = \frac{\eta - c}{K_2} \leq \frac{\eta}{K_2},
\]
so requiring $K_2 > \eta K_1/\delta$ completes the proof.
\end{proof}

\section{Applications of the potential game}
\label{sectionpotentialapplications}

We now use the potential game, and in particular Theorem \ref{theorempotentialHD}, to prove Theorems \ref{theoremarithbounds} and \ref{theoremBAdcodim}. Note that Theorem \ref{theoremKepsilondelta} follows immediately from combining Theorem \ref{theorempotentialHD} with Lemmas \ref{lemmaMepsilon} and \ref{lemmaBA} (cf. Remark \ref{remarkc0vsabsolute} and Example \ref{exampleabspoints}).

\subsection{Proof of Theorem \ref{theoremBAdcodim}}

Note: In this section we fix a norm on $\R^d$ and treat $\R^d$ as a metric space with respect to that norm, as well as letting $\BA_d(\epsilon)$ be defined in terms of this norm; it does not matter which norm it is.

\begin{lemma}
\label{lemmaBAd}
Let $\HH$ be the collection of hyperplanes in $\R^d$. Then for all $\epsilon > 0$ and $(d!V_d)^{1/d}\epsilon < \beta < 1$, the set $\BA_d(\epsilon)$ is $(\alpha,\beta,c,\rho,\HH)$-potential winning, where
\[
(\alpha,\beta,c,\rho,\HH) = \left(\frac{\epsilon\beta^{-1}}{(d!V_d)^{-1/d}-\epsilon\beta^{-1}},\beta,0,\beta(d!V_d)^{-1/d} - \epsilon,\HH\right).
\]
Here $V_d$ denotes the volume of the $d$-dimensional unit ball (with respect to the chosen norm).
\end{lemma}
When $d = 1$, these numbers are only slightly worse than the ones appearing in Lemma \ref{lemmaBA}.
\begin{proof}
As in the proof of Lemma \ref{lemmaBA}, we let
\[
\Delta_\epsilon(\pp/q) = B(\pp/q,\epsilon q^{-\frac{d+1}{d}})
\]
so that
\[
\BA_1(\epsilon) = \R\butnot\bigcup_{\pp/q\in\Q^d} \Delta_\epsilon(\pp/q).
\]
We will use the simplex lemma in the following form:
\begin{lemma}[Simplex Lemma, {\cite[Lemma 4]{KTV}}]
Fix $Q > 1$ and $s > 0$ such that
\begin{equation}
\label{simplexvolume}
V_d s^d = \frac{1}{d! Q^{d+1}}\cdot
\end{equation}
Fix $\xx\in\R^d$. Then the set
\begin{equation}
\label{simplex}
\{\pp/q\in \Q^d \cap B(\xx,s) : q < Q\}
\end{equation}
is contained in an affine hyperplane.
\end{lemma}
We now describe Alice's strategy in the potential game. Suppose that Bob has just made the move $B_\turn = B(\xx_\turn,\rho_\turn)$, and let $Q = Q_\turn > 1$ and $s = s_\turn > 0$ be chosen so as to satisfy \eqref{simplexvolume} as well as the equation
\[
s = \rho_\turn + \epsilon\beta^{-1} Q^{-\frac{d+1}{d}}.
\]
Note that solving for $\rho_\turn$ in terms of $Q$ gives
\begin{equation}
\label{rhon}
\rho_\turn = \left(\frac{1}{\sqrt[d]{d! V_d}} - \epsilon\beta^{-1}\right) Q^{-\frac{d+1}{d}}.
\end{equation}
Then Alice deletes the $\alpha\rho_\turn$-neighborhood of the affine hyperplane containing the set \eqref{simplex}.

To show that this strategy is winning (it is clearly legal), let $\xx$ denote the outcome of the game and suppose that $\xx\notin \BA_d(\epsilon)$, so that $\xx\in \Delta_\epsilon(\pp/q)$ for some $\pp/q\in\Q^d$. We will show that $\xx\in A\in \AA_\turn$ for some $\turn\geq 0$. Indeed, let $\turn$ be the first integer such that $q < Q_\turn$. If $\turn > 0$, then
\[
\beta \leq \frac{\rho_{\turn}}{\rho_{\turn-1}} = \left(\frac{Q_{\turn}}{Q_{\turn-1}}\right)^{-\frac{d+1}{d}}
\]
and thus
\[
q \geq Q_{\turn - 1} \geq \beta^{\frac{d}{d+1}}Q_\turn
\]
while if $\turn = 0$, then
\[
1 \leq \frac{\rho_0}{\rho} = \frac{Q_0^{-\frac{d+1}{d}}}{\beta}
\]
and thus
\[
q \geq 1 \geq \beta^{\frac{d}{d+1}}Q_\turn.
\]
Either way we have $q \geq \beta^{\frac{d}{d+1}}Q_\turn$, so
\begin{equation}
\label{radDeltaeps}
\rad(\Delta_\epsilon(\pp/q)) = \epsilon q^{-\frac{d+1}{d}} \leq \epsilon \beta^{-1} Q_\turn^{-\frac{d+1}{d}}.
\end{equation}
Thus since $\xx\in B(\xx_\turn,\rho_\turn)\cap \Delta_\epsilon(\pp/q)$, we have
\[
|\pp/q - \xx_\turn| \leq \rho_\turn + \epsilon \beta^{-1} Q_\turn^{-\frac{d+1}{d}} = s
\]
i.e. $\pp/q\in B(\xx_\turn,s)$. Thus $\pp/q$ is a member of the set \eqref{simplex} and thus of the hyperplane that Alice deleted the $\alpha\rho_\turn$-neighborhood of on turn $\turn$. So to complete the proof it suffices to show that
\[
\epsilon q^{-\frac{d+1}{d}} \leq \alpha\rho_\turn,
\]
which follows from \eqref{rhon}, \eqref{radDeltaeps}, and the definition of $\alpha$.
\end{proof}

\begin{corollary}
\label{corollaryKepsilondeltaeta}
Let $J \subset \R^d$ be the topological support of an Ahlfors $\delta$-regular and absolutely $\eta$-decaying measure. Then for all $\epsilon > 0$, we have
\begin{equation}
\label{Kepsilondelta}
\HD(\BA_d(\epsilon)\cap J) \geq \delta - K \epsilon^\eta
\end{equation}
where $K$ is a constant depending on $J$.
\end{corollary}
\begin{proof}
Let $\beta = 1/4$ and $c = \eta/2$. Combining Lemma \ref{lemmaBAd} with Proposition \ref{propositionmonotonicity2} shows that $\BA_d(\epsilon)$ is $(\alpha,\beta,c,\rho,\HH)$-potential winning, where $\rho$ is a constant, $\alpha \asymp \epsilon$, and $\HH$ is the collection of hyperplanes in $\R^d$. If $\epsilon$ is sufficiently small, then $\alpha^c \leq \frac1{K_2}(1-\beta^{\eta-c})$ and thus Theorem \ref{theorempotentialHD} shows that \eqref{Kepsilondelta} holds.
\end{proof}

Theorem \ref{theoremBAdcodim} is a special case of of this corollary (cf. Example \ref{example1decay}).

\subsection{Proof of Theorem \ref{theoremarithbounds}}

In this section we let $\PP$ denote the set of points in $X = \R$.

\begin{lemma}
\label{AP-general}
For all $0 < \beta \leq 1/4$, there exists $\delta = \delta(\beta)$ such that for all $\alpha,c,\rho,\epsilon > 0$ and $S \subset \R$ such that $\w S = S \cup (-\infty, a) \cup (a + 2\rho + \epsilon, \infty) \subset \R$ is an $(\alpha,\beta, c, \rho,\PP)$-potential winning set with $c \leq 1 - 1/\log(\alpha^{-1})$, the set $S$ contains an arithmetic progression of length $\delta\alpha^{-1}/\log(\alpha^{-1})$. In fact, for every sufficiently small $t > 0$, $S$ contains uncountably many arithmetic progressions of length $\delta\alpha^{-1}/\log(\alpha^{-1})$ and common gap size $t$.
\end{lemma}
\begin{proof}
By Proposition \ref{propositionmonotonicity2}, we may without loss of generality assume that $c = 1 - 1/\log(\alpha^{-1})$. Fix $k\in\N$ to be determined, and fix $0 < t \leq \epsilon/k$. By Propositions \ref{propositionintersection2} and \ref{propositioninvariance2}, the set
\[
S' = \bigcap_{i = 0}^{k-1} (\w S - it)
\]
is $(k^{1/c}\alpha,\beta,c,\rho,\PP)$-potential winning. Thus by Theorem \ref{theorempotentialHD}, if
\begin{equation}
\label{etsAP-general}
k\alpha^c \leq \frac{1}{K_2}(1-\beta^{1-c})
\end{equation}
then $\AD(S'\cap [a,a+2\rho]) > 0$. In particular, in this case $S'\cap [a,a+2\rho] \neq \emptyset$, and if $x\in S'\cap [a,a+2\rho] \neq \emptyset$ then the arithmetic progression $\{x,x+t,\ldots,x+(k-1)t\}$ is contained in $S$.

Now let $k$ be the largest integer such that \eqref{etsAP-general} is satisfied. To complete the proof, we need to show that $k \asymp \alpha^{-1}/\log(\alpha^{-1})$ as long as $\alpha$ is sufficiently small. Indeed, since $\beta$ is fixed and $c = 1-1/\log(\alpha^{-1})$, we have
\begin{align*}
1-\beta^{1-c} &= 1-\beta^{1/\log(\alpha^{-1})} \asymp 1/\log(\alpha^{-1}),&
\alpha^c &= e\alpha \asymp \alpha
\end{align*}
and thus
\[
k = \left\lfloor \frac1{K_2} \cdot \frac{1-\beta^{1-c}}{\alpha^c}\right\rfloor \asymp \frac{1-\beta^{1-c}}{\alpha^c} \asymp \frac{1/\log(\alpha^{-1})}{\alpha} = \frac{\alpha^{-1}}{\log(\alpha^{-1})}
\]
as long as the right-hand side large enough to guarantee that $k\geq 1$.
\end{proof}

Combining with Lemmas \ref{lemmaMepsilon} and \ref{lemmaFn} (cf. Remark \ref{remarkc0vsabsolute}) immediately yields the lower bounds of \eqref{arithbounds1} and \eqref{arithbounds2}, respectively. So in the remainder of the proof we will demonstrate the upper bounds.

Let $S$ be an arithmetic progression in $M_\epsilon$ of length $k\geq 2$, and let $I$ be the smallest interval appearing in the construction of $M_\epsilon$ such that $S \subset I$. Let $J$ be the middle $\epsilon$ gap of $I$. The minimality of $I$ implies that $S$ contains points both to the left and to the right of $J$, so the common gap size $t$ of $S$ is at least $|J| = \epsilon |I|$. On the other hand, we have $(k-1)t = \diam(S) \leq |I|$, so $k-1 \leq |I|/|J| = 1/\epsilon$. This demonstrates the upper bound of \eqref{arithbounds1}.

The proof for $F_n$ is similar but more technical. In what follows we use the standard notation
\begin{equation}
\label{cfracdef}
[a_0;a_1,a_2,\ldots] \df a_0+\cfrac1{a_1+\cfrac1{a_2+\ddots}}
\end{equation}
Let $S$ be an arithmetic progression in $F_n$ of length $k\geq 2$, and let $\omega = \omega_1\cdots \omega_r$ be the longest word in the alphabet $\{1,\ldots,n\}$ such that the continued fraction expansions of all elements of $S$ begin with $\omega$. (Note that $\omega$ may be the empty word.) Then the set $A$ of numbers $i = 1,\ldots,n$ such that some element of $S$ has a continued fraction expansion of the form $[0;\omega,i,\ldots]$ has at least two elements. Here $[0;\omega,i,\ldots]$ is short for $[0;\omega_1,\ldots,\omega_r,i,\ldots]$. Let $i$ and $j$ be the smallest and second-smallest elements of $A$, respectively, and consider first the case where $j = i+1$. As before, write $t$ for the common gap size of $S$, so that $(k-1)t = \diam(S)$. Then
\[
t \geq |[0;\omega,j,n+1] - [0;\omega,i,1]| \;\;\;\;\; \text{ while } \;\;\;\;\;
(k-1)t \leq |[0;\omega,i] - [0;\omega,n+1]|,
\]
so
\begin{align*}
k-1 &\leq \frac{|[0;\omega,i] - [0;\omega,n+1]|}{|[0;\omega,j,n+1] - [0;\omega,i,1]|}\\
&\asymp \frac{|[0;i] - [0;n+1]|}{|[0;j,n+1] - [0;i,1]|} \note{bounded distortion property\footnotemark}\\
&\leq \frac{1/i}{|[0;j,n+1] - [0;j]|}\\
&\asymp \frac{1/i}{(1/j^2)|[0;n+1] - 0|} \note{bounded distortion property again}\\
&= \frac{j^2}{i}(n+1) \lesssim n^2. \since{$j = i + 1$}
\end{align*}
\Footnotetext{The bounded distortion property for the Gauss iterated function system $\big(u_k(x) \df \frac1{k+x}\big)_{k\in\N}$ can be proven by applying \cite[Lemma 2.2(a)]{MauldinUrbanski1}. It states that if $u_\omega(x) = u_{\omega_1}\circ\cdots\circ u_{\omega_r}(x)$, or equivalently $u_\omega([0;x]) = [0;\omega,x]$, then
\[
|u_\omega(y) - u_\omega(x)| \asymp \max_{[0,1]} |u_\omega'| \cdot |y-x| \text{ for all }x,y\in [0,1].
\]}
If $j > i+1$, then the bound $|[0;j,n+1] - [0;i,1]| \geq \frac{1}{i+1} - \frac{1}{j}$ can be used instead, yielding the better bound
\begin{align*}
k-1 &\lesssim \frac{\frac{1}{i}}{\frac{1}{i+1}-\frac{1}{j}}
\leq \frac{\frac{1}{i}}{\left(\frac{1}{(i+1)(i+2)}\right)}
= \frac{(i+1)(i+2)}{i} \asymp i \leq n.
\end{align*}
This demonstrates the upper bound of \eqref{arithbounds2}, completing the proof.

\bibliographystyle{amsplain}

\bibliography{bibliography}

\end{document}